\documentclass[11pt,letterpaper]{amsart}

\usepackage{amsmath}
\usepackage{amssymb} 
\usepackage{fancyvrb}

\usepackage{stmaryrd} % double brackets

\newcommand{\Z}{{\mathbb Z}} 
\newcommand{\Q}{{\mathbb Q}} 
\newcommand{\C}{{\mathbb C}}

\newcommand{\D}{{\mathbb D}} 

\newcommand{\cO}{{\mathcal O}}

\newcommand{\cG}{{\mathcal G}}

\newcommand{\cris}{{\textup{cris}}} 

\DeclareMathOperator{\End}{End}
\DeclareMathOperator{\Gal}{Gal}

\newcommand{\gp}{{\mathfrak p}}

\newcommand{\go}{{\mathfrak o}}
\newcommand{\gr}{{\mathfrak r}}
\newcommand{\ga}{{\mathfrak a}}

\newtheorem{thm}{Theorem}
\newtheorem*{KCT} {Theorem A (Katz)}
\newtheorem{prop}{Proposition}
\newtheorem{lemma}{Lemma}

\newtheorem{cor}{Corollary}

\theoremstyle{remark}

\newtheorem{rem}[]{Remark}

\newcommand{\itop}[2]{\genfrac {}{}{0pt}{3}{#1}{#2} }

\begin{document}

\author[P. Guerzhoy]{P. Guerzhoy}
\address{ 
Department of Mathematics,
University of Hawaii, 
2565 McCarthy Mall, 
Honolulu, HI,  96822-2273 
}
\email{pavel@math.hawaii.edu}

\title[]{On the $p$-adic values of the weight two Eisenstein series for supersingular primes}
\keywords{elliptic curves, $p$-adic modular forms, formal groups}
\subjclass[2020]{11F33, 14H52, 11S31}
\begin{abstract}
%Let $E=\C/\Lambda$ be an elliptic curve. Among the values of modular forms at $(E,dz)$, the value of the weight $2$ Eisenstein series, call it  $A(\Lambda)$, is especially interesting. % and attracted the attention of mathematicians since Eisenstein and Ramanujan. 
%Assume that $E$ is defined over an algebraic number field $H$, and let $\gp$ be a place of $H$ such that $E$ has good reduction at $\gp$. We offer a definition for the $\gp$-adic value $A^{(\gp)}(\Lambda) \in \hat{H}_\gp$ of this series. Our definition applies uniformly to the cases of ordinary and supesingular reduction. In the case of ordinary reduction, although our definition is different from that given by Katz in \cite{Katz_RAnEis, Katz_properties}, the value $A^{(\gp)}(\Lambda)$ is the same. If the elliptic curve $E$ has complex multiplication (therefore $A(\Lambda) \in H$), we prove that  $A^{(\gp)}(\Lambda) = A(\Lambda)$ both in the case of  supersingular  and, of course, ordinary reduction.
The weight two Eisenstein series may be considered as the first example of a Katz $p$-adic modular form. Classically, its values are defined for the primes of ordinary reduction. 
We investigate a modified definition which applies uniformly to all primes of good reduction, both ordinary and supersingular. 
We show that, in the case of complex multiplication, these $p$-adic values coincide with the algebraic value of this Eisenstein series.
\end{abstract}

%\subjclass[2010]{11F33, 11G07}
%\keywords{elliptic modular forms, congruences}
%\begin{abstract}
%For an elliptic defined curve over an algebraic number field, and a prime $\gp$ of this field where the curve has good (ordinary or supersinugular) %reduction, we suggest a way definition We suggest a definition for 
%\end{abstract}
\maketitle

\section {Introduction}

The Ramanujan series 
\[
P =1 -24\sum_{n \geq 1} q^n \sum_{d | n} d \hspace{5mm} \text{($q=\exp(2 \pi i \tau)$ with $\Im(\tau)>0$ throughout)}
\]
narrowly misses being a modular form. The closely related function  standardly denoted by $E_2^*(\tau)$ and defined by
\[
E_2^*(\tau) = P(\tau) - \frac{3}{\pi \Im(\tau)}
\]
is modular of weight $2$ (and trivial level) and can be considered as an algebraic modular form over $\C$ in the sense of \cite[2.1.1]{Katz_RAnEis}. 
We recall the construction now. Consider a  lattice $\Lambda \subset \C$. The elliptic curve $E=\C/\Lambda$ with nowhere-vanishing invariant differential $dz$ 
admits an affine  Weierstrass model
\begin{equation} \label{eq_we}
E \ : \ y^2=4x^3-g_2x-g_3,
\end{equation}
with  nowhere-vanishing invariant differential 
\[
\omega = dx/y,
\]
where 
\[
g_2= g_2(\Lambda)=60 \sum_{\itop{m \in \Lambda}{m \neq 0}} m^{-4}, \hspace{4mm} 
g_3= g_3(\Lambda) =140 \sum_{\itop{m \in \Lambda}{m \neq 0}} m^{-6}.
\]
Then $\omega$ and $\bar{\omega}$ form a $\C$-basis  of $H^1_{DR}(E/\C)$ and one has
for $\eta := xdx/y \in H^1_{DR}(E/\C)$ that
\[
\eta + A\omega + B\bar{\omega} =0 \hspace{4mm} \textup{in  $H^1_{DR}(E/\C)$.}
\]
Equivalently, the function on the complex plane $\C$
\begin{equation} \label{eq_zeta_A_B}
\zeta(z;\Lambda) -Az-B\bar{z} \hspace{4mm} \textup{is $\Lambda$-periodic},
\end{equation}
where $\zeta(z;\Lambda)$ is the Weierstrass $\zeta$-function associated with the lattice $\Lambda$.
The quantity  $A=A(\Lambda)$ is thus defined as the coefficient of $\omega$ in the decomposition of $\eta$ with respect to the basis of $H^1_{DR}(E/\C)$ which consists of $\omega$ and a (anti-holomorphic) differential $\nu$ characterized by the property 
\begin{equation} \label{eq_non_hol}
\nu(cz) = \bar{c}\nu(z).
\end{equation}
(The space of such differentials $\nu$ is one-dimensional and $\bar{\omega}$ is one of them.)
This way, the quantity $A=A(\Lambda)$  becomes the value of a weight $2$ modular form (on $SL(2,\Z)$) defined over $\C$. 
%in the sense of \cite[2.1.1]{Katz_RAnEis}.
The relationship between the two modular forms
\[
A(\Lambda)= \frac{\pi^2}{3\omega_1^2} E_2^*(\omega_2/\omega_1)
\]
%A(\Lambda)=\frac{1}{12} \frac{4\pi^2}{\omega_1^2} E_2^*(\omega_2/\omega_1)
for any basis $\Lambda = \langle \omega_1,\omega_2 \rangle$ with $\Im(\omega_2/\omega_1)>0$ 
%is established by Katz 
follows from \cite[Lemma 1.3.6]{Katz_RAnEis}.

While  $A(\Lambda)$ is an algebraic modular form over $\C$, the function $E_2^*(\tau) $ is manifestly non-holomorphic and (see \cite[2.4]{Katz_RAnEis})  $A(\Lambda)$  is not an algebraic modular form over $\Z$ as the $q$-expansion of $P$ would suggest. In particular, there is no reason to expect the quantity $12A(\Lambda) \in \C$ to belong to $\Z[g_2,g_3]$, or even be an algebraic number whenever $g_2$ and $g_3$ are algebraic.  There is a notable exception: the number $A(\Lambda)$ becomes algebraic together with $g_2$ and $g_3$  whenever the elliptic curve $E$ has complex multiplication (see \cite[Lemma 4.0.7]{Katz_RAnEis} or, alternatively, for an elementary short proof, \cite[Lemma 3.1]{Masser} ). From now on, we assume that the elliptic curve $E$ is defined over an algebraic number field $H$ (i.e. $g_2,g_3 \in H$) which we assume to be a normal extension of $\Q$ by passing to the normal closure of $\Q(g_2,g_3)$ if necessary. 

For a rational prime $p>3$, Ramanujan's series $P(q)$ is the $q$-expansion of a $p$-adic modular form in the sense of Katz \cite[A2.4]{Katz_properties}, \cite[Lemma 5.7.8]{Katz_RAnEis}. 
Specifically, assume that $p$ is not ramified in $H$. % ground ring must be (p503) complete and separated i.e. $A \simeq \lim_{leftarrow} A/p^nA$ in $p$-adic topology https://www.imo.universite-paris-saclay.fr/~fontaine/galoisrep.pdf 
For a prime $\gp$ of $H$ above $p$, assume that $E$ has good {\sl ordinary} reduction at $\gp$. Denote by 
\[A^{(\gp)}(\Lambda):= P(E,\omega) \in \hat{H_\gp}
\]
 the value of this $p$-adic modular form  (see \cite[5.10.6]{Katz_RAnEis}) which is defined as an element of the field of fractions of the $\gp$-adic completion $\hat{\cO_\gp}$ of $\cO_\gp$, the localization at $\gp$ of the ring of integers $\cO \subset H$. 

In the case when the elliptic curve $E$  has complex multiplication, we thus have independently defined quantities $A(\Lambda) \in H \subset  \hat{H_\gp}$  %(independent on an embedding $H \hookrightarrow \C$) 
and $A^{(\gp)}(\Lambda) \in \hat{H_\gp}$. These two quantities in fact coincide, which is the essential special case (see the proof of  \cite[Lemma 8.0.13]{Katz_RAnEis}) of the Comparison Theorem 
 \cite[Theorem 8.0.9]{Katz_RAnEis} which we now quote.

 %%%%%%%%%%%%%%%%%%%%%%%%%%%%%%%%%%%%%%%%%%%%%%%%%%%%%%%% {KCT} -- Katz theorem
 \begin{KCT} \label{KCT}

 Assume that the elliptic curve $E$ admits a Weierstrass equation \textup{(\ref{eq_we})} with $g_2,g_3 \in \cO$ and has complex multiplication by an order in an imaginary quadratic field 
 $K \subseteq H$. Let $\gp$ be a place of $H$ such that $E$ has good ordinary reduction at $\gp$. Then 
 \[
 A^{(\gp)}(\Lambda) = A(\Lambda).
 \]
 In particular,  the value 
  $A^{(\gp)}(\Lambda) \in \hat{H_\gp}$
in fact lies in $H \subset \hat{H_\gp}$ and is independent of $\gp$. 
 
 \end{KCT}
 
We now  describe %an explicit 
another way to define the $p$-adic quantity $A^{(\gp)}(\Lambda) \in \hat{H_\gp}$. %We assume that $E$ has a Weierstrass equation (\ref{eq_we}) with $g_2,g_3 \in \cO_\gp$. 
 We do not assume any complex multiplication for the time being. 
 %We assume that the extension $H/\Q$ is normal, and that $p$ is unramified in $H$; the former condition is always achievable by extending $H$ if necessary.
We assume that $p>2$ is unramified in $H$ and that (\ref{eq_we}) has good ({\sl either ordinary or supersingular}) reduction at $\gp$ (in particular, $g_2,g_3 \in \cO_{\gp}$).
The formal group law $\hat{E}$ associated with $E$ (see (\ref{eq_fg_def}) for the definition)  is defined over $\cO_{\gp}$. %$\cO[1/2]$. 
Let 
\[
l(t) = \sum_{n \geq 1} \frac{b(n)}{n} t^n \hspace{4mm} \text{with $b(1)=1$}
\] 
be the logarithm of a formal group law $F$ strictly isomorphic to $\hat{E}$ over $\cO_\gp$. 
 For $b \in H \subset \hat{H_\gp}$, we write $b^\phi$ for the action of Frobenius (see Section \ref{sec_FG} for details).
Let 
\[
l^\phi(t) := \sum_{n \geq 1}  \frac{b(n)^\phi}{n} t^n.
\]

It is known (see Proposition \ref{prop_lambda_mu} below for details) that there exist $\lambda, \mu \in \hat{\cO_\gp}$ such that the formal power series
\begin{equation} \label{eq_zeta_lambda_mu}
\zeta(l(t); \Lambda) - \lambda l(t) - \mu \frac{1}{p} l^\phi(t^p) \in \frac{1}{t} + \hat{\cO_\gp}\llbracket t \rrbracket
\end{equation}
has all $\gp$-integer coefficients. Moreover, one can and we will always assume that $\mu=0$ whenever the elliptic curve $E$ is ordinary at $\gp$. Under this assumption, the quantities $\lambda$ and $\mu$ are defined uniquely and we shall show that they do not depend on the choice of a formal group law $F$ (see (\ref{eq_zeta_F_lambda_mu}) below). 
Define $A^{(\gp)}(\Lambda) \in  \hat{H_\gp}$ to be the coefficient $\lambda$ in (\ref{eq_zeta_lambda_mu}):
%%%%%%%%%%%%%%%%%%%%%%%%%%%%%%%%%%%%%%%%% Definition
%\begin{dfn} \label{definition}
%Let $E$ be an elliptic curve defined over an algebraic number field $H$.
%Let $p>2$ be a prime unramified in $H$, and $\gp$ a prime of $H$ above $p$ such that $E$ has good (ordinary or supersingular) reduction at $\gp$.
%Choose a Weierstrass equation \textup{(\ref{eq_we})} for $E$ such that $g_2,g_3 \in \cO_\gp$, and let $\omega=dx/y$.
%Set
\begin{equation} \label{eq_def}
 A^{(\gp)}(\Lambda) := \lambda.
\end{equation}
%\end{dfn}
The  definition is motivated by the parallelism between (\ref{eq_zeta_lambda_mu}) and (\ref{eq_zeta_A_B}). 
A discussion leading to this definition can be found in \cite[Section 4]{BKY}, where the parallelism between the functions $\bar{z}$ and $l^\phi(t^p)/p$  is observed. 
This definition introduces the quantity  $A^{(\gp)}(\Lambda) \in \hat{\cO_\gp}$ for all but finitely many places $\gp$.
Definition (\ref{eq_def}) differs substantially from the original Katz definition even in the ordinary reduction case; in particular, it involves neither trivialization nor consideration of a canonical subgroup. 
It follows from an observation of Mazur and Tate \cite[Proposition 3]{MT} that our  $A^{(\gp)}(\Lambda)$ coincides with the value of the $p$-adic Eisenstein series of weight $2$ originally defined by Katz whenever the Weierstrass model 
(\ref{eq_we}) 
has good ordinary reduction at $\gp$. 
%Besides this observation from \cite{MT}, 
%our definition is motivated by the parallelism between (\ref{eq_zeta_lambda_mu}) and (\ref{eq_zeta_A_B}). 
%Note also that a discussion closely related  to our definition can be found in \cite[Section 4]{BKY}, where the parallelism between the functions $\bar{z}$ and $l^\phi(t^p)/p$ for CM elliptic curves is observed. 

\begin{rem}

Let us indicate some directions to generalize definition (\ref{eq_def}) which we do not pursue in this paper primarily because the reduction of the involved group 
$\hat{E}$ has height one, and everything can be formulated and proved within the original framework of \cite{Katz_properties,Katz_RAnEis}.

One may  incorporate the case of multiplicative reduction. Namely, one can start with a consideration of $(E,\omega)$ defined by (\ref{eq_we}) instead of a lattice $\Lambda$. That allows for consideration of a singular cubic (with a node). 

One may also consider elliptic curves over other rings. For example, the consideration of the Tate elliptic curve $\textup{Tate}(q)$ defined over the ring $\Z((q))$ yields 
$\lambda=\frac{1}{12}P(q) \in \frac{1}{12}\Z\llbracket q \rrbracket$, the Ramanujan series, in order to satisfy  (\ref{eq_zeta_lambda_mu}) for all primes $p>3$.

\end{rem}

There is of course a great deal of indeterminacy in choosing $\Lambda$ such that $E=\C/\Lambda$ satisfies all requisite conditions: if it does, then so does $\Omega\Lambda$ for any $\Omega \in \cO^*_{\gp}$. It is easy to see that $ A^{(\gp)}(\Lambda)$ behaves as a weight $2$ modular form
\[
 A^{(\gp)}(\Omega\Lambda) = \Omega^{-2} A^{(\gp)}(\Lambda) \hspace{3mm} \textup{ for $\Omega \in \cO^*_{\gp}$} 
\]
with respect to these transformations. 

%Definition (\ref{eq_def}) is redeemed by our main result which extends Theorem A above and which we will state now.

Let now $E$ have complex multiplication. Specifically, let $K$ be an imaginary quadratic extension of rationals $\Q$, and let $\go_K \subset K$ be its ring of integers. 
The endomorphism ring $\go= \End(E)$ is an order  $\go \subset \go_K$. Let $\ga$ be a proper $\go$-ideal which we may think of as a lattice  $\ga \subset \C$.  Let $H=K(j(\ga))$ denote the ring class field of $\go$, and choose $\Omega \in \C^*$ such that  $g_2,g_3 \in \cO_\gp \subset H$, where
%%%%%%%%%%%%%%%%%%%%%%%%%%%%%%%%%%%%%%%%%    {g_Omega}
\begin{equation} \label{g_Omega}
g_2= g_2(\Omega \ga)= \Omega^{-4} g_2(\ga) %\frac{60}{\Omega^4} \sum_{\itop{m \in \ga}{m \neq 0}} m^{-4}, 
\hspace{4mm} 
g_3= g_3(\Omega \ga) = \Omega^{-6} g_3(\ga). %\frac{140}{\Omega^6} \sum_{\itop{m \in \ga}{m \neq 0}} m^{-6}.
\end{equation}
With  $\Lambda=\Omega\ga$, we have $E=\C/\Lambda \simeq_{\C} \C/\ga$ with affine model  (\ref{eq_we}) 
%With these $g_2$ and $g_3$, we have the Weierstrass equation (\ref{eq_we}) of $E \simeq \C/\ga$ 
(note that any elliptic curve with complex multiplication admits such a description by 
\cite[Proposition 4.8] {Shimura_arith}).  We thus have the complex number $A=A(\Lambda)$ determined by (\ref{eq_zeta_A_B}) which we can by \cite[Lemma 4.0.7]{Katz_RAnEis} and will consider as an element of $H$. We also have   $A^{(\gp)}(\Lambda) \in \hat{H_\gp}$ introduced in definition (\ref{eq_def}).

The main result of this paper is the following extension of Theorem A to the case of supersingular reduction.

%%%%%%%%%%%%%%%%%%%%%%%%%%%%%%%%%%%%%%%%% {th_main} MAIN THEOREM
\begin{thm} \label{th_main}
Let $\gp$ be a place of $H$ such that \textup{(\ref{eq_we})} has good reduction at $\gp$. 
In particular, we assume that $\gp$ lies above a rational prime $p > 3$ which divides neither the discriminant of $K$ nor the conductor $f=[\go_K :\go]$.
Then 
 \[
 A^{(\gp)}(\Lambda) = A(\Lambda).
 \]
 In particular, the value 
  $A^{(\gp)}(\Lambda) \in \hat{H_\gp}$
in fact lies in $\cO_{\gp} \subset H \subset \hat{H_\gp}$ and is independent of $p$. 

 \end{thm}

The proof of Theorem A (see \cite[Lemma 8.0.13]{Katz_RAnEis} for details) is based on the decomposition 
\[
H^1_{DR}(E \otimes \hat{\cO_{\gp}}/\hat{\cO_{\gp}}) \otimes \hat{H_{\gp}} \simeq H_{DR}^1(E/H) \otimes_K \hat{K}_{\gp} = U_\gp \oplus U'_\gp
\]
into the eigenspaces for the action of the $\left| \cO_\gp/\gp \right|$-th power Frobenius endomorphism. This decomposition is  combined with the observation that, for the unit root eigenspace $U_\gp$, the intersection $U_\gp \cap H^1_{DR}(E/K)$ coincides with the $\bar{c}$-eigenspace for $c \in \go = \End(E)$ ($c \notin \Z$). The parallelism with the decomposition $H^1_{DR}(E/\C) \simeq \langle \omega \rangle \oplus \langle \bar{\omega} \rangle$ with $\bar{\omega}$ satisfying  (\ref{eq_non_hol}) is employed in the proof. 
The unit root eigenspace exists  when the elliptic curve $E$ has ordinary reduction at $\gp$ (a generalization in terms of the canonical subgroup is available only for those  supersingular elliptic curves whose Hasse invariant is ``not too near zero", see \cite[Section 3 and Appendix 2]{Katz_properties}). 

Our approach to the proof of Theorem \ref{th_main} in Section \ref{sec_proof_main} is quite different while parallel. Instead of the two-dimensional vector space $H^1_{DR}$ as above we consider (see \eqref{eq_dieud} below for the definition) the $\cO_\gp$-module $\D(F) \simeq H^1_\cris(F)$, where $F$ is any formal group law isomorphic to $\hat{E}$ over $\cO_\gp$. 
This module has rank $1$ in the case when the reduction of $E$ at $\gp$ is ordinary, and rank $2$ when the reduction is supersingular. In the latter case, which is the case of primary interest for us, the basis $\{l(t), l^\phi(t^p)/p\}$ turns out to be an eigenbasis for $c \in \go = \End(E) \hookrightarrow \End(F)$ ($c \notin \Z$) with eigenvalues $c$ and $\bar{c}$ correspondingly. We lay out the details in Section \ref{sec_FG}.  It is now in the supersingular reduction case when we employ the parallelism with the decomposition $H^1_{DR}(E/\C) \simeq \langle \omega \rangle \oplus \langle \bar{\omega} \rangle$ with $\bar{\omega}$ satisfying  (\ref{eq_non_hol}) to obtain our result.

In Section \ref{sec_modular_proof}, we present an alternative approach to the proof of Theorem \ref{th_main} which is based on the modularity of the elliptic curves with complex multiplication. We can obtain only a slightly weaker result on this way. 
Assume that there exists a non-trivial map $\phi: X_1(N) \rightarrow E = \C/\Lambda$ defined over $H$. The modular parametrization $\phi$ exists when $E$ is a factor of $J_1(N)$ as in  \cite{Shimura71}. Then one can assume that $\phi$ takes the cusp $i \infty$ to the neutral element of $E$, and let $g=\sum b(n) q^n$ be defined by $\phi^*\omega = q^{-1}g(q)dq$. Observe that $g$ is a CM cusp form (which in general is not a Hecke eigenform), and $b(1)\neq 0$. The renormalization $\Lambda \mapsto b(1)^{-1} \Lambda$ allows us to assume that $b(1)=1$. We associate a formal group law $\cG$ to $g$ (see (\ref{eq_lG_def})), and prove (Theorem \ref{th_honda_gen}) that, for all but finitely many primes $\gp$ of $H$, there is an isomorphism
of the formal group laws $\cG \simeq \hat{E}$ over $\cO_{\gp}$. (Theorem \ref{th_honda_gen} may be considered as a variation on the theme of a result of Honda  \cite[Theorem 5]{Honda}.)

%The existence of $\phi$  allows us to prove a variant of a theorem of Honda \cite{Honda} (Theorem \ref{th_honda_gen}): the formal group law whose logarithm is the Eichler integral of $g$ is weakly isomorphic over $\cO_\gp$ to the formal groups law $\hat{E}$. 
%For modular forms, character twists are available. 
We build on the character twist idea as it is used in a similar situation in \cite[proof of Theorem 1.3]{BOR}.
%We 
Specifically, consider the pull back to $X_1(N)$ of  (\ref{eq_zeta_A_B}) and its twist with the quadratic character associated with the extension $K/\Q$. 
Exploiting the CM property of $g$ allows us to obtain the result (Theorem \ref{th_main_modular}).

%%%%%%%%%%%%%%%%%%%%%%%%%%%%%%%%%%%%%%%%%%%%%%%%
%%%%%%%%%%%%%%%%%%%%%%%%%%%%%%%%%%%%%%%%%%%%%%%%
%%%%%%%%%%%%%%%%%%%%%%%%%%%%%%%%%%%%%%%%%%%%%%%%
\section*{Acknowledgements}

The author is very grateful to Barry Mazur and Ken Ono for their  inspiring and encouraging  remarks.
The author thanks the referee for the remarks which helped the author to improve the presentation.
In particular, following a suggestion of the referee, the author supplied the arXiv version of the manuscript with an appendix 
%(Section \ref{sec_app}) 
which contains and explains a code producing the specific numerical evidence described in Section \ref{sec_ex}.

\section{Example} \label{sec_ex}

The numerical calculations were carried out using PARI/GP \cite{PARI2}.

Let $K=\Q(\sqrt{-15})$, and 
let $E$ be the elliptic curve with complex multiplication by $\go_K$ with $j=-85995w - 52515$ such that $w^2-w-1=0$. The curve $E$ is defined over $\Q(j)=\Q(w)=\Q(\sqrt5)$ by 
%%%%%%%%%%%%%%%%%%%%%%%%%%%%%%%%%%%%%%%%%%%%%%%         {eq_we_example}
\begin{equation} \label{eq_we_example}
y^2=4x^2-g_2x-g_3
\end{equation}
with
\[
g_2=11505w + 7110 \hspace{4mm} \textup{and} \hspace{4mm} g_3=356720w + 220465.
\]
The norm of the discriminant is $$\mathcal{N} (\Delta)= \mathcal{N}(-68647176000w - 42426288000) = 2^{12}3^65^6,$$ and the elliptic curve (\ref{eq_we_example}) has good reduction for every prime $\gp$ of the Hilbert class field $H=\Q(\sqrt{-15},\sqrt{5})$ lying above a prime $p>5$.

The two embeddings $K \hookrightarrow \C$ yield two values of $\tau$ in the (boundary of the) fundamental domain in the upper half-plane. These are $\tau = (1+\sqrt{-15})/4$ corresponding to $w \mapsto (1-\sqrt{5})/2$ and $\tau=(-1+\sqrt{-15})/2$ corresponding to $w \mapsto (1+\sqrt{5})/2$.

In either case, we have
\[
A= \frac{1}{12} \frac{4\pi^2}{\omega_1^2}E_2^*(\tau) = \frac{1}{12}(126w + 78) = \frac{21w+13}{2}.
\]

The formal group $\hat{E}$ associated with the Weierstrass equation (\ref{eq_we_example}) is defined over  $\go_{\Q(\sqrt5)}\left[\frac{1}{2}\right]$. Specifically, we pick $t=-2x/y$ as the parameter and compute the formal group logarithm

\begin{multline*}
l_{\hat{E}}(t) = \int\frac{dx}{y} = 
t + \left(-711 - \frac{2301}{2}w \right)t^5 + \left( -\frac{94485}{4} - 38220w \right)t^7 \\
+ \left( \frac{60972375}{8} + \frac{98655375}{8}w \right)t^9 + \left( \frac{1288993125}{2} + \frac{4171269375}{4}w \right)t^{11} \\
+ \left( -\frac{200868706875}{2} - 162506197500w \right)t^{13} + \mathcal{O}(t^{15}).
\end{multline*}

The Weierstrass $\zeta$-function has the expansion
\begin{multline*}
\zeta(z) = \frac{1}{z} + \left(-\frac{237}{2} - \frac{767}{4}w \right)z^3 + \left(-\frac{6299}{4} - 2548w\right)z^5 \\ +  \left(-\frac{2438895}{112} - \frac{563745}{16}w\right)z^7  + \left(-\frac{2455225}{8} - \frac{7945275}{16}w \right)z^9 \\ + \left(-\frac{1517389435}{352} - 6974965w\right)z^{11}  + \left(-\frac{25264737675}{416} - \frac{3144554175}{32}w\right)z^{13} \\ +  \mathcal{O}(z^{15}).
\end{multline*}

Theorem A combined with \cite[Proposition 3]{MT} implies that 
\[
\zeta(l_{\hat{E}}(t) ) -  \frac{21w+13}{2} l_{\hat{E}}(t)  \in \frac{1}{t} + \cO_\gp \llbracket t \rrbracket
\]
for every place $\gp$ of $H$ lying above a prime $p>5$ which splits in $K$. The coefficients of this series belong to $\Q(\sqrt5)$, and a computer calculation shows that the 
first thousand coefficients are indeed $p$-integral for all $p>5$ such that $\left( \frac{-15}{p} \right)=1$.

Theorem \ref{th_main} predicts that, for every place $\gp$ of $H$ lying  above $p>5$ such that $\left( \frac{-15}{p} \right)=-1$, there exists $\mu \in \hat{\cO}_\gp$
such that 
%%%%%%%%%%%%%%%%%%%%%%%%%%%%%%%%%%%%%%%%%%%%%%% {eq_pasif}
\begin{equation} \label{eq_pasif}
\zeta(l_{\hat{E}}(t) ) -  \frac{21w+13}{2} l_{\hat{E}}(t)  -\mu \frac{1}{p}  l^\phi_{\hat{E}}(t^p) \in \frac{1}{t}+ \cO_\gp \llbracket t \rrbracket.
\end{equation}
In other words, the theorem claims that, while $\mu  \in \hat{\cO}_\gp$ depends on $\gp$, the quantity $(21w+13)/2$ is universal for all primes of good reduction. 

Note that $l_{\hat{E}}(t) \in \Q(w) \llbracket t \rrbracket$ and the Artin map acts on its coefficients simply as the unique non-trivial element of $\Gal\left( \Q(w)/\Q \right)$ sending $w$ to $1-w$. 

For example, let $p=7$. Choosing $\mu \equiv 47 \pmod {7^2}$ we compute the first thousand coefficients of the series (\ref{eq_pasif}) and observe that they are all $7$-integral (the prime $7$ is inert in $\Q(w)$).

%%%%%%%%%%%%%%%%%%%%%%%%%%%%%%%%%%%%%%%%%%%%%%%%
%%%%%%%%%%%%%%%%%%%%%%%%%%%%%%%%%%%%%%%%%%%%%%%%
%%%%%%%%%%%%%%%%%%%%%%%%%%%%%%%%%%%%%%%%%%%%%%%%

\section{Preliminaries} \label{sec_FG}

In this section, we recall briefly the definition of and some basic facts about $\D(F) \simeq H^1_\cris(F)$ for a one-dimensional formal group $F$.  We refer to \cite[Section 2]{BKY} for a detailed survey. All formal group laws under consideration in this paper are one-dimensional, and we suppress this qualifier. We then %specialize to the situation of interest for us, and 
investigate the behavior of a certain basis of $\D(F)$ under a formal group law isomorphism (Proposition \ref{prop_eigenbasis}). In conclusion, we make a simple but useful remark related to the action of the Artin symbol.

Let $p$ be a rational prime. 
Throughout, we assume that $H$ is a finite normal extension of $\Q$, and that $p$ is not ramified in $H$. As in the Introduction, we denote by $\cO \subset H$ the ring of integers, by $\cO_\gp \supset \cO$ the valuation ring,  by $\hat{\cO_\gp}$ its completion, and by $\hat{H_\gp}$ the field of fractions of $\hat{\cO_\gp}$. Then $\hat{H_\gp}$ is a Galois extension of $\Q_p$, with a cyclic Galois group $\Gal (\hat{H_\gp}/\Q_p )$ generated by the Frobenius automorphism. The image of this Galois group under the inclusion 
\[
\Gal\left(\hat{H_\gp}/\Q_p \right) \hookrightarrow  \Gal(H/\Q)
\]
is the decomposition group $D_\gp \subset \Gal(H/\Q)$ generated by the image of Frobenius which is the Artin symbol, and we will denote by $b^\phi$ the action of Frobenius on $b \in \hat{H_\gp}$ with 
\[
b^\phi = \left( \frac{H/\Q}{\gp} \right)(b) \hspace{4mm} \textup{whenever $b \in H \subset \hat{H_\gp}$.}
\]

%%%%%%%% Weston, A brief Introduction to local fields ... refers to "Local Fields" by Serre ..

For a formal power series $\xi \in \hat{H_\gp} (( u ))$, we define the action coefficient-wise, namely for $\xi=\sum_n a(n)u^n$, we write
\[
\xi^\phi:= \sum_{n} a(n)^\phi u^n.
\]

The embeddings $\cO \hookrightarrow \cO_\gp \hookrightarrow \hat{\cO_\gp}$ allow us to consider 
a  formal group law $F=F(X,Y)$ defined over $\cO$ or $\cO_\gp$ as a formal group law over $\hat{\cO_\gp}$. Such a formal group law is determined by its logarithm which we write as
\[
l_F=l_F(u) = \sum_{n \geq 1} \frac{b_F(n)}{n} u^n \hspace{4mm} \text{with $b_F(1)=1$ and $b_F(n) \in \cO_\gp$}.
\]
For such a formal group law, we define following \cite{Katz_crys}, %$H^1_\cris(F)$ is defined as 
\begin{equation} \label{eq_dieud}
\D(F):= Z(F)/B(F),
\end{equation}
where 
\[
Z(F):= \left\{ \xi \in u\hat{H_\gp} \llbracket u \rrbracket  \Big\vert \  \xi' \in \hat{\cO_\gp} \llbracket u \rrbracket \hspace{1mm} \text{and} \hspace{1mm} \xi(F(u,v)) - \xi(u) - \xi(v) \in \hat{\cO_\gp}\llbracket u,v \rrbracket \right\},
\]
where $\xi'$ is the usual formal derivative of the power series $\xi  \in u\hat{H_\gp} \llbracket u \rrbracket$,
and
\[
B(F):=  u\hat{\cO_\gp} \llbracket u \rrbracket \subseteq Z(F).
\]
%\[
%H^1_\cris(F) := \frac
%{\left\{ \xi \in u\hat{H_\gp} \llbracket u \rrbracket \ \vert \  \xi' \in \hat{\cO_\gp} \llbracket u \rrbracket \hspace{2mm} \text{and} \hspace{2mm} \xi(F(u,v)) - \xi(u) - \xi(v) \in \hat{\cO_\gp}\llbracket u,v \rrbracket \right\}}
%{\left\{ \xi \in u\hat{\cO_\gp} \llbracket u \rrbracket \right\} }.
%\]
The map $\xi \mapsto d\xi:=\xi' du$ establishes an isomorphism of the $\hat{\cO_\gp}$-modules $\D(F) \simeq  H^1_\cris(F)$; the latter $\hat{\cO_\gp}$-module is considered in \cite{BKY}. 
%We stick to our notations avoiding introduction of a pairing in $H^1_\cris(F)$; that allows us to simplify the exposition avoiding any usage of %pairings and 
We stick to $\D(F)$ since we need neither a pairing on  $H^1_\cris(F)$ nor
the Colmez integration (see \cite[Section 3]{BKY} for a survey). 

We denote by $\tilde{\xi} \in  \D(F) $ the element represented by $\xi \in Z(F) $.

From now on, the modulo $\gp$ reduction of any formal group law under consideration is assumed to be of finite height.

It is proved in \cite[Theorem 5.3.3]{Katz_crys} that $\D(F)$ is a free $\hat{\cO_\gp}$-module whose rank $h=h(F,\gp)$ is equal to the height of the modulo $\gp$ reduction of $F$. 
A completely different proof which yields an explicit basis for this module is presented in  \cite[Proposition 2.3(iii)]{BKY}. We adopt this result to our notations and state it in the following proposition.

%%%%%%%%%%%%%%%%%%%%%%%%%%%%%%%%%%%%%%%%%%%%%%%%% Prop {prop_lambda_mu}
\begin{prop} \label{prop_lambda_mu}
%Assume that $h_F=1$ or $2$. 
For a formal group law $F$ over $\cO_\gp$ whose modulo $\gp$ reduction has height  $h=h(F,\gp)$,
the elements
\[
\ell_{F,1}:=\tilde{l}_F(u) \hspace{3mm} \textup{and} \hspace{3mm} \ell_{F,m} := \frac{1}{p} \tilde{l}_F^{\phi^{m-1}}(u^{p^{m-1}}) \hspace{3mm} \textup{with $m=2, \ldots, h$}
\]
form a basis of $\D(F)$.

\end{prop}

In particular let $h(F,\gp)=1$ or $2$. Then, for a power series $\xi \in Z(F)$, there exist unique (assuming $\mu=0$ whenever $h(F,\gp)=1$) quantities $\lambda,\mu \in \hat{\cO_\gp} $ such that
%%%%%%%%%%%%%%%%%%%%%%%%%%%%%%%%%%%%%%%%%%%%%%%%% {eq_xi_lambda_mu}
\begin{equation} \label{eq_xi_lambda_mu}
\xi(u) - \lambda l_F (u)- \mu \frac{1}{p} l_F^\phi(u^p) \in \hat{\cO_\gp} \llbracket u \rrbracket.
\end{equation}

%where 
%\[
%l^\phi_F(v) = \sum_{n \geq 1} \frac{b_F(n)^\phi}{n} v^n.
%\]

Let $F$ and $G$ be  formal group laws over $\cO_\gp$ with their logarithms $l_F$ and $l_G$ correspondingly, and let $t : G \rightarrow F$ be a homomorphism (over $\cO_\gp$). Then $t \in u \cO_\gp\llbracket u \rrbracket$ and it is easy to see (e.g. \cite[Corollary IV.4.3]{Silverman}) that 
\begin{equation} \label{eq_hom_log}
l_F(t(u)) = t'(0)l_G(u).
\end{equation}
The quantity 
\[
\alpha : = t'(0) \in \cO_\gp
\]
determines the homomorphism uniquely since $t(u)=l_F^{-1}(\alpha l_G(u))$.
More generally, we have a homomorphism of $\hat{\cO_\gp}$-modules $t^* :  \D(F) \rightarrow \D(G)$ well-defined by 
\[
(t^*\xi)(u) := \xi(t(u))
\]
for a representative $\xi \in Z(F)$  of an element  $\tilde{\xi} \in \D(F)$.
The following proposition is obviously true when $h(F,\gp)=h(G,\gp)=1$ in view of (\ref{eq_hom_log}). 

%%%%%%%%%%%%%%%%%%%%%%%%%%%%%%%%%%%%%%%%%%%%%%%%%%%%%%%%% Proposition {prop_eigenbasis}
\begin{prop} \label{prop_eigenbasis}
%Let $h_F=h_G=2$.  
%The bases $\{l_{G,1},l_{G,2}\}=\{l_G(u),\frac{1}{p}l_G^\phi (u^p) \}$ of $H^1_\cris(G)$ and  $\{l_{F,1},l_{F,2}\}=\{l_F(u),\frac{1}{p}l_F^\phi (u^p) \}$ of $H^1_\cris(F)$ diagonalize $t^*$. Specifically, the equalities 

The bases $\{ \ell_{F,m} \}_{m=1, \ldots, h(F,\gp)}$ and $\{ \ell_{G,m} \}_{m=1, \ldots, h(G,\gp)}$ diagonalize $t^*$.
Specifically, 
\[
t^*\ell_{F,m} = \alpha^{\phi^{m-1}} \ell_{G,m} \hspace{4mm} \textup{for $m=1,\ldots,h(G,\gp)$}
\]

\end{prop}

%The proposition remains true as stated when $\alpha \in \gp$. However, 
We are interested only in the case when $\alpha \in \cO_\gp^*$ for our applications. Note that the assumption that  $\alpha \in \cO_\gp^*$ is a unit implies that $h(F,\gp)=h(G,\gp)$. Indeed, otherwise it is easy to see (\cite[Section 2.2.3]{Lubin}) that the modulo $\gp$ reduction of $t$ must be zero, and $t \in u\gp\llbracket u\rrbracket$ implies $\alpha \in \gp$.

Before proving Proposition \ref{prop_eigenbasis}, we state an immediate corollary for later use. 

%%%%%%%%%%%%%%%%%%%%%%%%%%%%%%%%%%%%%%%%%%%%%%%%%%%%%%%%% Corollary {cor_stability}
\begin{cor} \label{cor_stability}

Assume that $h(F,\gp)=h(G,\gp)=1$ or $2$.

Let $\xi \in Z(F)$, and let $\lambda$ and $\mu$ be as in \textup{(\ref{eq_xi_lambda_mu})}.

\begin{itemize}

\item[(a)]
Let $t: G \rightarrow F$ be a strict isomorphism, that is $\alpha=t'(0)=1$. Then 
\[
t^*\xi - \lambda l_G (u)- \mu \frac{1}{p} l_G^\phi(u^p) \in \hat{\cO_\gp} \llbracket u \rrbracket,
\]

\item[(b)]
Let $t: F \rightarrow F$ be a weak endomorphism, that is $\alpha=t'(0) \in \cO_\gp^*$ is a unit. Then

\[
t^*\xi - \alpha\lambda l_F (u)- \alpha^\phi\mu \frac{1}{p} l_F^\phi(u^p) \in \hat{\cO_\gp} \llbracket u \rrbracket.
\]

\end{itemize}

\end{cor}

\begin{proof}

Substitute $t(u) \in u \cO_\gp \llbracket u \rrbracket$ for $u$ into (\ref{eq_xi_lambda_mu}), and use Proposition \ref{prop_eigenbasis}.

\end{proof}

Our proof of Proposition \ref{prop_eigenbasis} will employ the following adaptation to our notations of \cite[Lemma 4]{Honda} which we will prove here for the convenience of the reader.

%%%%%%%%%%%%%%%%%%%%%%%%%%%%%%%%%%%%%%%%%%%%%%%%%%%%%%%%% Lemma {lemma_honda}
\begin{lemma} \label{lemma_honda}
Let $A,B \in \hat{\cO_\gp} \llbracket u \rrbracket$, and assume that $A-B \in \gp \hat{\cO_\gp} \llbracket u \rrbracket$.
Then, for $n \geq 1$, 
\[
\frac{A^n-B^n}{n} \in  \gp \hat{\cO_\gp} \llbracket u \rrbracket.
\]
\end{lemma}

\begin{proof}
Write $n=n_0p^\nu$ with $p \nmid n_0$. The statement is obvious for $\nu=0$. Since $(A^{n_0p^\nu}-B^{n_0p^\nu})/(A^{p^\nu}-B^{p^\nu}) \in \hat{\cO_\gp} \llbracket u \rrbracket$, it suffices to prove that, for  $\nu \geq 1$,
\[
\frac{A^{p^\nu}-B^{p^\nu}}{p^\nu}  \in \gp \hat{\cO_\gp} \llbracket u \rrbracket.
\]
Since $p$ is unramified, we can write $B= A+ p C$ with $C \in   \hat{\cO_\gp} \llbracket u \rrbracket$. It suffices to prove that for $1 \leq m \leq p^\nu$ 
\[
\binom{p^\nu}{m}p^{m-\nu} \equiv 0 \pmod p.
\]
That follows from
\begin{multline*}
\binom{p^\nu}{m}p^{m-\nu}= p^\nu(p^\nu-1) \ldots (p^\nu-m+1) \frac{1}{m!}p^{m-\nu} \\ = (p^\nu-1) \ldots (p^\nu-m+1) \frac{p^m}{m!} \equiv 0 \pmod p
\end{multline*}
because $\textup{ord}_p(m!) = (m-S_m)/(p-1) <m$, where $S_m>0$ is the sum of digits of $m$ written in base $p$.

\end{proof}

%%%%%%%%%%%%%%%%%%%%%%%%%%%%%%%%%%%%%%%%%%%%%%%%%%%%%%%%%%%

\begin{proof}[Proof of Proposition \ref{prop_eigenbasis} ]

For $m=1$, the equality $t^*l_{F,1}=\alpha l_{G,1}$ follows immediately from (\ref{eq_hom_log}) (and implies the proposition in the case when $h(F,\gp)=h(G,\gp)=1$ as it was already observed). We thus have to prove that, for $k=m-1 \geq 1$,
\[
\frac{1}{p}l_F^{\phi^k}(t(u)^{p^k}) - \alpha^{\phi^k} \frac{1}{p} l_G^{\phi^k}(u^{p^k}) \in \cO_{\gp} \llbracket u \rrbracket.
\]
Apply the Artin symbol to  (\ref{eq_hom_log}) coefficient-wise $k$ times to obtain
\[
l_F^{\phi^k}(t^{\phi^k}(u)) = \alpha^{\phi^k} l_G^ {\phi^k} (u).
\]
Observe that
\begin{multline*}
l_F^{\phi^k}(t(u)^{p^k})-\alpha^{\phi^k} l_G^{\phi^k}(u^{p^k}) = l_F^{\phi^k}(t(u)^{p^k})- l_F^{\phi^k}(t^{\phi^k}(u^{p^k})) \\ =\sum_{n \geq 1} b_F(n)^{\phi^k} \ 
\frac{ t(u)^{p^kn} - t^{\phi^k}(u^{p^k})^n }{n}.
\end{multline*}
We  now apply Lemma \ref{lemma_honda} with $A=t(u)^{p^k}$ and $B=t^{\phi^k}(u^{p^k})$ (the congruence $A-B \in \gp \cO_\gp \llbracket u \rrbracket$ can be obtained by induction in $k$) to conclude the proof.
\end{proof}

%%%%%%%%%%%%%%%%%%%%%%%%%%%%%%%%%%%%%%%%%%%%%%%%%%%%%%%%%%%

Let now $E$ be an elliptic curve defined over $H$ which admits the affine Weierstrass model (\ref{eq_we}) with $g_2,g_3 \in \cO_\gp$, and assume that $p>2$.
We denote by $\hat{E}$ the formal group law associated with $E$ with respect to the parameter $u=-2x/y$. Specifically, let
\[
x=\wp(l)=\frac{1}{l^2}+\sum_{n \geq 2} c_n l^{2n-2} \in \frac{1}{l^2}+\Q[g_2,g_3]\llbracket l \rrbracket,
\]
where 
\[
c_2=\frac{1}{20} g_2, \hspace{3mm} c_3=\frac{1}{28}g_3, \hspace{3mm} c_n=\frac{3}{(n-3)(2n+1)}\sum_{m=2}^{n-2} c_mc_{n-m} \hspace{3mm} \textup{for $n \geq 4$}
\]
is the Laurent expansion around zero of the Weierstrass $\wp$-function. Then the power series 
\[
u(l)= -2x/y=-2\wp(l)/\wp'(l) = l + \ldots \in l\Q[g_2,g_3]\llbracket l \rrbracket
\]
has its leading coefficient $1$. We invert this power series setting 
\[
l_{\hat{E}}(u):=u+ \ldots \in  u\Q[g_2,g_3]\llbracket u \rrbracket,
\]
and define 
%%%%%%%%%%%%%%%%%%%%%%%%%%%%%%%%%%%%%%%%%%%%%%%%%%%%%%%%% {eq_fg_def}
\begin{equation} \label{eq_fg_def}
\hat{E}(X,Y)=u(l_{\hat{E}}(X)+l_{\hat{E}}(Y))
\end{equation}
 to be the formal group law whose logarithm is $l_{\hat{E}}$. A minor adaptation of the argument provided in \cite[Section IV.1]{Silverman} allows us to conclude that the formal group law $\hat{E}$ is actually defined over the ring $\Z[g_2/4,g_3/4]\subseteq \cO_\gp$ while 
%%%%%%%%%%%%%%%%%%%%%%%%%%%%%%%%%%%%%%%%%%%%%%%%%%%%%%%%%%% {eq_wp_wpprime_integral}
\begin{equation} \label{eq_wp_wpprime_integral}
u^2x=u^2\wp(l(u))  \in  1+ \cO_\gp \llbracket u \rrbracket  \hspace{3mm} \textup{and} \hspace{3mm} u^3y=u^3\wp'(l(u)) \in -2+  \cO_\gp\llbracket u \rrbracket . 
\end{equation}

We now consider the expansion around zero of the Weierstrass $\zeta$-function associated with $E$
\[
\zeta(l) = \frac{1}{l} - \sum_{n \geq 2} \frac{c_n}{2n-1}l^{2n-1}. 
\]
Define a formal power series $\psi(u)$ by
\[
\zeta(l_{\hat{E}}(u))=\frac{1}{u} +d_0 + \psi(u), 
\]
where the constant term $d_0 \in \cO_\gp$ is chosen in a way such that $\psi(u) \in  uH\llbracket u \rrbracket$. It follows from \cite[Lemma 3.4]{BKY} (see also \cite[Proposition 12]{Gue_adele} for an alternative, though less elegant argument) that $\psi \in Z(\hat{E})$ assuming $p \geq 3$, and since now on we stick to this assumption. Assume that the elliptic curve $E$ has good reduction at $\gp$. Then the height of the modulo $\gp$ reduction of $\hat{E}$ is either $1$ or $2$ corresponding to the ordinary and supersingular reduction cases. Corollary \ref{cor_stability}(a) allows us to rewrite (\ref{eq_xi_lambda_mu}) as 
%%%%%%%%%%%%%%%%%%%%%%%%%%%%%%%%%%%%%%%%%%%%%%%%%%%%%%%%%%% {eq_zeta_F_lambda_mu}
\begin{equation} \label{eq_zeta_F_lambda_mu}
\zeta(l_F(u)) - \lambda l_F(u) - \mu \frac{1}{p} l_F^\phi (u^p) \in \frac{1}{u} + \hat{\cO}_\gp \llbracket u \rrbracket
\end{equation}
with $\lambda$ and $\mu$ independent of the choice of a 
formal group law $F$ with logarithm $l_F$ which is strictly isomorphic to $\hat{E}$ (over $\cO_\gp$).

%%%%%%%%%%%%%%%%%%%%%%%%%%%%%%%%%%%%%%%%%%%%%%%%

%%%%%%%%%%%%%%%%%%%%%%%%%%%%%%%%%%%%%%%%%%%%%%%%

We finish this section with a remark on the action of the Artin symbol on elements of the imaginary quadratic field $K$. Let $d_K<0$ be the discriminant of $K$, and  let $\go$ be an order in $K$. Let now $H$ be the ring class field of $\go$. It is well-known (e.g. \cite[Lemma 9.3]{Cox}) that $H$ is a Galois extension of $\Q$ with with Galois group 
$\Gal(H/\Q) \simeq \Gal(H/K) \rtimes (\Z/2\Z)$, specifically one has
\[
1 \rightarrow \Gal (H/K) \rightarrow \Gal(H/\Q) \rightarrow \Gal(K/\Q) \rightarrow 1.
\]
In particular, every $\sigma \in \Gal(H/\Q)$ decomposes in a unique way as $\sigma = \sigma_1 \sigma_2$ with $\sigma_1 \in  \Gal (H/K)$ and $\sigma_2 \in \Gal (K/\Q)$, and, for $b \in K \subset H$, we can calculate $\sigma(b) = \sigma_2(b)$. 

In particular,  for an element $a + b \sqrt{d_K}$ we have
%%%%%%%%%%%%%%%%%%%%%%%%%%%%%%%%%%%%%%%%%%%%%%%% {eq_artin_action}
\begin{equation} \label{eq_artin_action}
\begin{split}
(a+b \sqrt{d_K})^\phi := \left( \frac{H/\Q}{\gp} \right)(a+b \sqrt{d_K}) \\ = a+b  \left( \frac{H/\Q}{\gp} \right) (\sqrt{d_K}) = a + b \left( \frac{d_K}{p} \right)  \sqrt{d_K}.
\end{split}
\end{equation}
The last equality is implied by the fact that, for any $\sigma_2 \in \Gal (K/\Q) \simeq \Z/2\Z$, we have $\sigma_2(\sqrt{d_K} )= \pm \sqrt{d_K}$, and, for the Artin map, 
it must be that
\[
 \left( \frac{H/\Q}{\gp} \right)  (\sqrt{d_K}) \equiv \left(\sqrt{d_K}\right)^p \pmod \gp,
\]
while 
\[
d_K^{\frac{p-1}{2}} \equiv  \left( \frac{d_K}{p} \right) \pmod p
\]
by the definition of the  Legendre symbol.
%%%%%%%%%%%%%%%%%%%%%%%%%%%%%%%%%%%%%%%%%%%%%%%%
%%%%%%%%%%%%%%%%%%%%%%%%%%%%%%%%%%%%%%%%%%%%%%%%
%%%%%%%%%%%%%%%%%%%%%%%%%%%%%%%%%%%%%%%%%%%%%%%%

%%%%%%%%%%%%%%%%%%%%%%%%%%%%%%%%%%%%%%%%%%%%%%%%
%%%%%%%%%%%%%%%%%%%%%%%%%%%%%%%%%%%%%%%%%%%%%%%%
%%%%%%%%%%%%%%%%%%%%%%%%%%%%%%%%%%%%%%%%%%%%%%%%

\section{Proof of Theorem \ref{th_main}} \label{sec_proof_main}

Recall that $H$ is the ring class field of an order $\go$ in an imaginary quadratic field $K$ of discriminant $d_K<0$ and $\ga$ is a proper $\go$-ideal.

For the elliptic curve $E=\C/\Lambda$ with $\Lambda=\Omega \ga$ 
%$E=\C/\Lambda  \simeq_\C \C/\ga$, 
%which satisfies the conditions of Theorem \ref{th_main}, 
we assume that $\Omega$ in (\ref{g_Omega}) is chosen such that 
$g_2,g_3 \in \cO_\gp$ and  affine Weiersttrass model (\ref{eq_we}) has good reduction at $\gp$. (Note that $1728g_2^3/(g_2^3-27g_3^2)=j(\ga) \in H$, and any choice of $\Omega$ such that $g_1,g_2 \in H$ is admissible for all but  finitely many places $\gp$ of $H$.)

For $\alpha \in \go$, we denote by $[\alpha] \in \End(E)$ the corresponding endomorphism such that $[\alpha]^*\omega = \alpha \omega$, where the differential $\omega=dx/y$. We denote by $\hat{E}$ the formal group law adopted to the parameter $u=-2x/y$ defined  in (\ref{eq_fg_def}) by its logarithm 
$l_{\hat{E}}(u) \in uH\llbracket u \rrbracket$.

%%%%%%%%%%%%%%%%%%%%%%%%%%%%%%%%% basic notations repeated

The endomorphism $[\alpha]$ is defined over $H$ (see \cite[Theorem 2.2.2(b)]{Silverman_a}) and induces an endomorphism $t_\alpha \in \End(\hat{E})$ of the formal group law $\hat{E}$. Specifically, $t_\alpha \in uH\llbracket u \rrbracket$ is the unique power series such that 
%%%%%%%%%%%%%%%%%%%%%%%%%%%%%%%%% {eq_t_alpha}
\begin{equation} \label{eq_t_alpha}
l_{\hat{E}}(t_\alpha (u)) = \alpha l_{\hat{E}}(u).
\end{equation}
It is easy to see that, for all but possibly finitely many places $\gp$ of $H$, we have  $t_\alpha \in u\cO_\gp \llbracket u \rrbracket$. 
%is actually $\gp$-integral.
Indeed, the functions $\wp(\alpha z; \Lambda)$ and $\wp'(\alpha z; \Lambda)$ are $\Lambda$-periodic, therefore can be expressed as rational functions (over $H$) of
$x=\wp(z;\Lambda)$ and $y=\wp'(z;\Lambda)$. Thus 
\[
t_\alpha(u) = -2\wp(\alpha z;\Lambda)/\wp'(\alpha z;\Lambda) \in H(x,y),
\]
and substituting the series expansions (\ref{eq_wp_wpprime_integral}) into this rational function we support the claim. A more subtle analysis (essentially similar to but a bit more  involved than our proof of Proposition \ref{prop_zeta_alpha_zeta_integral} below) is carried out in \cite{Streng}. In particular, \cite[Theorem 2.9]{Streng} adopted to our setting states that
\[
t_\alpha(u) \in u\cO_\gp \llbracket u \rrbracket
\]
for all places $\gp$ of good reduction.

We now substitute $t_\alpha(u)$ for $u$ into (\ref{eq_zeta_F_lambda_mu})  with $l_F=l_{\hat{E}}$ and  use (\ref{eq_t_alpha}) along with Corollary \ref{cor_stability}(b) to obtain
\[
\zeta(\alpha l_{\hat{E}}(u);\Lambda) - \lambda \alpha l_{\hat{E}}(u) - \mu \alpha^\phi \frac{1}{p}l_{\hat{E}}(u^p) \in \frac{1}{\alpha u} + \hat{\cO}_\gp \llbracket u \rrbracket
\]
assuming that $ \alpha \notin \gp$. 
%Note that the Weierstrass $\zeta$-function is taken with respect to the chosen lattice, specifically, $\zeta(z)=\zeta(z;\Omega \ga)$.
%\begin{rem}
%Although we did not need that for further argument, note that this proves that $\zeta(\alpha l_{\hat{E}}(u)) \in Z(\hat{E})$.
%\end{rem}
Combining this equation with  (\ref{eq_zeta_F_lambda_mu}) (taken again for  $l_F=l_{\hat{E}}$) we conclude that
%%%%%%%%%%%%%%%%%%%%%%%%%%%%%%%%%%%%%%% {eq_fg_zeta_zeta}
\begin{equation} \label{eq_fg_zeta_zeta}
\zeta(\alpha l_{\hat{E}}(u);\Lambda) -  \alpha^\phi \zeta(l_{\hat{E}}(u);\Lambda) - \lambda( \alpha - \alpha^\phi) l_{\hat{E}}(u) \in \frac{\alpha^{-1}-\alpha^\phi}{u} + \hat{\cO}_\gp\llbracket u \rrbracket.
\end{equation}

%%%%%%%%%%%%%%%%%%%%%%%%%%%%%%%%%%%%%% {prop_zeta_alpha_zeta_integral} 

\begin{prop} \label{prop_zeta_alpha_zeta_integral} 

%For $A=A(\Lambda)$, 
One can choose $\alpha\in \go$ (e.g. $\alpha = f\sqrt{d_K}$ works just fine) such that for some non-zero integer $M \in \Z$,
\[
M \left( \zeta(\alpha l_{\hat{E}}(u);\Lambda) - \bar{\alpha} \zeta(l_{\hat{E}}(u);\Lambda) - A (\Lambda)(\alpha - \bar{\alpha}) l_{\hat{E}}(u) \right)   \in M\frac{\alpha^{-1}-\bar{\alpha}}{u} + \cO_\gp\llbracket u \rrbracket,
\]
where the bar denotes the complex conjugation.
\end{prop}

%\begin{rem}
%A finer analysis in the spirit of the proof of the proof of  \cite[Theorem 2.9]{Streng} allows one to eliminate the indeterminate $M$ gaining control over possible denominators. We do not need that for our purposes.
%\end{rem}

We postpone a proof of the proposition and complete the proof of Theorem \ref{th_main} now. 

Recall that $A^\gp(\Lambda)= \lambda$ by definition (\ref{eq_def}), so we need to prove that $\lambda=A(\Lambda)$.

Let us first consider the case when $E$ has good supersingular reduction at $\gp$, and the height of the modulo $\gp$ reduction of $\hat{E}$ is $2$. This happens exactly when $p$ is inert in $K$, and (\ref{eq_artin_action}) implies
\[
\alpha^\phi = \bar{\alpha}.
\] 
Combine Proposition \ref{prop_zeta_alpha_zeta_integral} with (\ref{eq_fg_zeta_zeta}) to obtain
\begin{equation} \label{eq_log_bounded}
M(\lambda - A(\Lambda))(\alpha - \bar{\alpha}) l_{\hat{E}}(u) \in \hat{\cO}_\gp\llbracket u \rrbracket.
\end{equation}
Since $l_{\hat{E}}(u) $ has infinitely many zeros, (because the height of the modulo $\gp$ reduction of $\hat{E}$ is assumed to be finite, namely $2$) and a choice of $\alpha \in \go$ with $\alpha \neq \ \bar{\alpha}$ which satisfies all extra conditions  is easily available (e.g. $\alpha = f\sqrt{d_K}$ works), this proves that $\lambda = A(\Lambda)$ as required.

The  case when $E$ has good ordinary reduction follows immediately from \cite[Proposition 3]{MT} combined with a special case of \cite[Theorem 8.0.9]{Katz_RAnEis} quoted above as Theorem A. That case actually served as a motivation and a starting point  for the author. However, in \cite{MT}, the equality $\lambda = A(\Lambda)$ follows from a much stronger result (the existence of the $p$-adic $\sigma$-function), and it may be interesting to see how it follows directly from our considerations. In this case, $p$ splits in $K$,  therefore  (\ref{eq_artin_action}) implies
\[
\alpha^\phi = \alpha,
\]
and (\ref{eq_fg_zeta_zeta}), though  true, becomes useless. However, since the height of the modulo $\gp$ reduction of $\hat{E}$ is $1$, we have (\ref{eq_zeta_F_lambda_mu}) with $\mu=0$. We substitute  $t_\alpha(u) \in u\cO_\gp \llbracket u \rrbracket $ for $u$ into (\ref{eq_zeta_F_lambda_mu})  and combine it with (\ref{eq_zeta_F_lambda_mu}) itself to obtain
\[
\zeta(\alpha l_{\hat{E}}(u);\Lambda) -  \bar{\alpha} \zeta(l_{\hat{E}}(u);\Lambda) - \lambda( \alpha - \bar{\alpha}) l_{\hat{E}}(u) \in \frac{\alpha^{-1}-\bar{\alpha}}{u} + \hat{\cO}_\gp\llbracket u \rrbracket,
\]
which we can again combine with Proposition \ref{prop_zeta_alpha_zeta_integral} to conclude that $\lambda = A(\Lambda)$ by the same reasoning as above.

In order to accomplish the proof of Theorem \ref{th_main} we now have only to prove Proposition \ref{prop_zeta_alpha_zeta_integral}.

\begin{proof}[Proof of  Proposition \ref{prop_zeta_alpha_zeta_integral}]

We 
%abbreviate $\Lambda=\Omega\ga$ and 
start with the $\Lambda$-periodic function (\ref{eq_zeta_A_B}) on the complex plane.
 Since $\alpha \Lambda \subseteq \Lambda$, the function $\zeta(\alpha z;\Lambda) - A\alpha z - B\bar{\alpha} \bar{ z}$ is also $\Lambda$-periodic. Combining these two functions we conclude that the meromorphic function 
\[
\zeta(\alpha z;\Lambda) - \bar{\alpha} \zeta(z;\Lambda) - A(\alpha - \bar{\alpha})z 
\]
on the complex plain is  $\Lambda$-periodic and odd.
% it has all first order poles in the fundamental parallelogram of $\Lambda$.
It follows that we have 
\[
\zeta(\alpha z;\Lambda) - \bar{\alpha} \zeta(z;\Lambda) - A(\alpha - \bar{\alpha})z = \wp'(z;\Lambda) \frac{P(\wp(z;\Lambda))}{Q(\wp(z;\Lambda))}
\]
with polynomials $P,Q \in \C[x]$ with $\deg P= \deg Q-1$ (otherwise the right side would not have a pole of order exactly one at infinity as the left side has). Consider a partial fraction decomposition for the rational function on the right:
%%%%%%%%%%%%%%%%%%%%%%%%%%%%%%%%%%%%%%%%%%%%%%%%% {eq_partial_fraction}
\begin{equation} \label{eq_partial_fraction}
\zeta(\alpha z;\Lambda) - \bar{\alpha} \zeta(z;\Lambda) - A(\alpha - \bar{\alpha})z = \wp'(z;\Lambda) \sum_m \frac{t_m}{\wp(z,\Lambda)-s_m}.
\end{equation}
Note that all poles of the function on the left of (\ref{eq_partial_fraction}) are simple and located at $\mathcal{N}(\alpha) = \alpha \bar{\alpha}$-division points. 

It is easy to pick $\alpha \in \go$ (e.g. $\alpha=f\sqrt{d_K}$ works) such that all $s_m$, which are the values of the Weierstrass $\wp$-function at the division points, are  $\gp$-integers (in an extension of $H$) by \cite[Theorem VII.3.4(a)]{Silverman}; the same is also true about the values of $\wp'(z,\Lambda)$ at these points.

We now calculate and equate the residues at these poles of the functions on the left and right sides of (\ref{eq_partial_fraction}) in order to conclude that $t_m \in \bar{\Q}$. Assume that we have a pole at $z=z_0$, and $z_0 \not\in \Lambda/2$ 
%(the calculations in the cases when $z_0 \in \Lambda/2$ are similar). 
Thus $\wp'(z_0;\Lambda) \neq 0,\infty$, and $\wp(z_0;\Lambda) = s_m \neq \infty$. The corresponding term in the right side is
\[
\frac{\wp'(z;\Lambda)t_m}{\wp(z;\Lambda) -s_m} = \frac{\wp'(z;\Lambda)t_m}{\wp'(z_0;\Lambda)(z-z_0) + \ldots }
\]
with the residue of $t_m$ at $z_0$.
For the left side, our assumption implies that the $\zeta(z;\Lambda)$ is holomorphic at $z_0$, while $\zeta(\alpha z;\Lambda)$ has a simple pole (as $\alpha z_0 \in \Lambda$) with a residue of
\[
\frac{1}{2\pi i} \oint_{C_1} \zeta(\alpha z;\Lambda) dz = \frac{1}{2\pi i}  \frac{1}{\alpha} \oint_{C_2} \zeta(v;\Lambda) dv =  \frac{1}{\alpha}, 
\]
where $C_1$ is a small contour around $z_0$ and $C_2$ = $\alpha C_1$ is a small contour around $\alpha z_0 \in \Lambda$. Thus  $t_m =1/\alpha$. 
A similar calculation in the case when  $z_0 \in \Lambda/2$ while $z_0 \not \in \Lambda$ (note that $\wp'(z_0, \Lambda) = 0$ in this case) yields the same $t_m =1/\alpha$ outcome.

We thus see that there exists non-zero $M \in \Z$ such that, after multiplying by $M$, all numerators on the partial fraction decomposition in the right of (\ref{eq_partial_fraction}) become algebraic integers, and observe that so are all quantities $s_m$.

We now multiply (\ref{eq_partial_fraction}) by $M$, substitute $l(u)$ for $z$,  and use (\ref{eq_wp_wpprime_integral}) to conclude that the series expansion in $u$ on the right has all $\gp$-integral coefficients (though possibly in some extension of $H$), while all coefficients in the expansion of the left-hand side belong to $H$. 
\end{proof}

%%%%%%%%%%%%%%%%%%%%%%%%%%%%%%%%%%%%%%%%%%%%%%%%
%%%%%%%%%%%%%%%%%%%%%%%%%%%%%%%%%%%%%%%%%%%%%%%%
%%%%%%%%%%%%%%%%%%%%%%%%%%%%%%%%%%%%%%%%%%%%%%%% MODULAR FORMS

\section{Modular forms approach} \label{sec_modular_proof}

In this Section, we indicate an alternative approach to Theorem \ref{th_main}. Although the results we can prove in this way are somewhat weaker, the interpretation based on modular parametrization of CM elliptic curves  which implies similar results is, in our opinion, interesting.

We start with some preliminary considerations. 
For a weight $2$ cusp form $g=\sum_{n \geq 1} b(n) q^n$ such that $b(1)=1$, denote by $\cG$ the formal group law defined by its logarithm
%%%%%%%%%%%%%%%%%%%%%%%%%%%%%%%%%%%%%%%%%%%%%%%% {eq_lG_def}
\begin{equation} \label{eq_lG_def}
l_\cG(u):= \sum_{n \geq 1} \frac{b(n)}{n} u^n.
\end{equation}
While in general we can say little about $\cG$ besides it is defined over a characteristic zero field $\Q(b(2),b(3), \ldots)$, some special cases are of particular interest.  
An important theorem of Honda \cite[Theorem 5]{Honda} (some details pertaining to the primes $2$ and $3$ are cleaned up later in \cite{Hill}) states that, for an elliptic curve over $\Q$, the formal group associated with its N\'eron model is strictly isomorphic over $\Z$ to the formal group associated with its $L$-series. In particular, assume that $g$ is a primitive Hecke eigenform with $b(n) \in \Z$. Eichler - Shimura theory associates to $g$ an elliptic curve $E_g$ over $\Q$, and Honda's theorem implies that the formal group associated with the  N\'eron model of $E_g$ is strictly isomorphic to $\cG$ over $\Z$. 

We adopt the definition of modularity from \cite{Mazur}: we say that an elliptic curve $E$ over an algebraic number field $H$ is {\sl modular} (over $H$) whenever for some $N$ there exists a non-trivial map 
\[
\psi : X_1(N) \rightarrow E
\]
defined over $H$. Note that $H$ does not need to be the minimal field of definition for $E$, and we will always assume that $H$ is a normal extension of $\Q$ taking the normal closure if necessary. A composition with translation allows us to always assume that the cusp at infinity maps to the neutral element of $E$.

Of course, $E_g$ is modular (and so is any rational elliptic curve), while there are  modular elliptic curves besides those which are rational. We now introduce notations in order to state a weaker (but much simpler and more general) version of Honda's theorem for modular elliptic curves.

%Denote by $\go_H$ the ring of integers of $H$, and 
Denote by $\cO \subset H$ the ring of integers, and let
\[
\gr:= \cO\left[ \frac{1}{M} \right]
\]
for some integer $M$. We will never specify $M$ and it suffices to observe that, for all but finitely many prime ideals $\gp$ of $\cO$, the localizations at $\gp$ of $\gr$ and $\cO$ coincide.

For a modular (over $H$) elliptic curve $E$ choose a model (\ref{eq_we}) with differential $\omega = dx/y$. Then the pull-back is 
%%%%%%%%%%%%%%%%%%%%%%%%%%%%%%%%%%%%%%%%%%%%%%%%%%%%% {eq_omega_lift}
\begin{equation} \label{eq_omega_lift}
\psi^*\omega = g(q) \frac{dq}{q}
\end{equation}
with a weight $2$ cusp form $g=\sum_{n \geq 1} b(n) q^n$. 
We now assume that the map $\psi$ is not ramified at the cusp $i \infty$, equivalently, $b(1) \neq 0$.
Change  equation (\ref{eq_we}) to
\begin{equation} \label{eq_we_b(1)}
y^2=4x^3 - b(1)^4 g_2 x - b(1)^6 g_3
\end{equation}
so that $\omega=dx/y$ pulls back to a cusp form (which we still denote by $g = \sum b(n) q^n$) with $b(1)=1$.

%%%%%%%%%%%%%%%%%%%%%%%%%%%%%%%%%%%%%%%%%%%%%%%%%%%%% {th_honda_gen}

\begin{thm} \label{th_honda_gen}

In the above notations, let $\hat{E}$ be the formal group law over $\gr$ associated with the Weierstrass equation \textup{(\ref{eq_we_b(1)})}.

The formal group law $\cG$ associated to $g$ as above  is strictly isomorphic to $\hat{E}$ over $\gr$ and therefore is defined over $\gr$.

\end{thm}

\begin{rem}
Even in the case when the elliptic curve $E$ is rational, one does not obtain the full strength of the Honda theorem in this way. While one can consider a  N\'eron model for $E$ instead of (\ref{eq_we}), one still cannot guarantee that $b(1) \neq 0$.% (this is the Manin constant).
\end{rem}

\begin{proof}

Consider $u=-2x/y$ as a function on $E$. Its pull-back $\psi^*u \in q + q^2\gr\llbracket q \rrbracket$ (enlarging $M$ if necessary) 
%is then is a rational function with coefficients in $\Q(j)$ in $j(q)$ and $j(q^N)$ Milne elliptic curves p186
is a strict isomorphism $ \cG \rightarrow \hat{E}$ of formal group laws since we have that $l_{\hat{E}}((\psi^*u(q)) = l_\cG(q)$ which follows  from the definition (\ref{eq_lG_def}) combined with (\ref{eq_omega_lift}).

\end{proof}

Examples of modular elliptic curves which are not rational come from Shimura's construction \cite{Shimura71} of complex multiplication elliptic curves $E$ which are factors of Jacobians $J_1(N)$ of modular curves $X_1(N)$.
We thus have in these cases a non-trivial morphism 
%%%%%%%%%%%%%%%%%%%%%%%%%%%%%%%%%%%%%%%%%%%%%%%%%%%%% {eq_modular_param}
\begin{equation} \label{eq_modular_param}
\psi : X_1(N) \rightarrow J_1(N) \rightarrow E,
\end{equation}
where the first arrow is the canonical embedding, and the second is the projection. 
A converse statement, \cite[Proposition 1.6]{Shimura_cl}, implies that whenever a complex multiplication elliptic curve $E$ appears as a factor of $J_1(N)$, the pull-back $g=\psi^*\omega$ of the invariant differential on $E$ is a linear combination of cusps forms 
whose Mellin transforms are $L$-functions associated with primitive Hecke characters of the imaginary quadratic field $K$. 
It follows from \cite[Theorem 4.8.2]{Miyake} that $d_K^2 \vert N$, where $d_K$ is the discriminant of $K$.
All these are CM-forms, and so is $g$ as their linear combination. Specifically, we have
%%%%%%%%%%%%%%%%%%%%%%%%%%%%%%%%%%%%%%%%%%%%%%%%%%%%% {eq_char_twist}
\begin{equation} \label{eq_char_twist}
\sum_{n>0} \left( \frac{d_K}{n} \right) b(n) q^n  = \sum_{n>0} b(n) q^n.
\end{equation}
%where $d_K$ is the discriminant of $K$. 
Furthermore, it follows from \cite[Lemma 4.1(i)]{Gonzales_Lario} that $b(1) \neq 0$.

As an example, the elliptic curves considered by Gross in \cite{Gross} (those with complex multiplication by the ring of integers $\go_K$ of an imaginary quadratic field $K$ with a prime discriminant $d_K<-3$) have modular parametrization (\ref{eq_modular_param}) (defined over the Hilbert class field of $K$) by \cite[Theorem 20.1.1]{Gross}. In this case, $N=d_K^2$, and one can even use $X_0(N)$ instead of $X_1(N)$.
A more general family of complex multiplication elliptic curves which have modular parametrization (\ref{eq_modular_param}) is described in \cite{Wortmann}.

%%%%%%%%%%%%%%%%%%%%%%%%%%%%%%%%%%%%%%%%%%%%%%%%%%%
We now state and prove a version of Theorem \ref{th_main} for the CM elliptic curves which are modular in the above sense.

\begin{thm} \label{th_main_modular}
Assume that $E$ be a CM elliptic curve  which admits modular parametrization \textup{(\ref{eq_modular_param})} defined over  $H$ (which is normal  over $\Q$).  Assume furthermore  that $E$ has a model 
 \textup{(\ref{eq_we})} defined over $H$. 
 
 Then for all but possibly finitely many primes $\gp$ of $H$, 
 \[
 A^{(\gp)}(\Lambda) = A(\Lambda).
 \]
In particular, the value $A^{(\gp)}(\Lambda) \in \hat{H_\gp}$
in fact lies in $\cO_\gp \subset H \subset \hat{H_\gp}$ and is independent of $\gp$. 

\end{thm}

\begin{rem}

While it is possible to make the argument more precise at some points, the  employment of modular parametrization does not allow the author to entirely get rid of the possibility of exceptional primes of good reduction; Theorem \ref{th_main} is more accurate.

%However, this method is more flexible, and allows for generalizations where the weight is higher than $2$.

\end{rem}

\begin{proof}

Corollary \ref{cor_stability}(a) allows for the freedom to choose any formal group law $F$ isomorphic to $\hat{E}$ over $\cO_\gp$ and its logarithm $l_F$ in equation (\ref{eq_zeta_lambda_mu}) which defines the quantity $A^{(\gp)}(\Lambda)$ according to definition (\ref{eq_def}). Theorem \ref{th_honda_gen} allows us to make the choice $F=\cG$ with the logarithm $l=l_{\cG}$ associated with the modular form $g$ defined by  the modular parametrization in (\ref{eq_omega_lift}).

Recall that for a function $f$ on $X_1(N)$ and a primitive Dirichlet character $\chi$ modulo $r$ the twist $f \otimes \chi$ is
defined by
\[
f \otimes \chi := W(\bar{\chi})^{-1}\sum_{u=1}^r f(\tau+u/r),
\]
where $W(\bar{\chi})$ is the Gauss sum associated with the complex conjugate Dirichlet character $\bar{\chi}$. 
Obviously,
%%%%%%%%%%%%%%%%%%%%%%%%%%%%%%%%%%%%%%%%%%%    {eq_series_twist}
%\begin{equation} \label{eq_series_twist}
\[
\left(\sum a(n) q^n \right) \otimes \chi = \sum \chi(n)a(n) q^n,
\]
%\end{equation}
and we use that to  extend the definition of character twist to formal powers series.
%specifically 
%$\left(\sum a(n) t^n \right) \otimes \chi := \sum \chi(n)a(n) t^n$.

Let $\chi=\left( \frac{d_K}{\cdot} \right)$ be the quadratic character associated with $K$.
Since $d_K^2 \vert N$, it follows from \cite[Proposition 3.64]{Shimura_arith} that 
 $f \otimes \chi$ is a function on $X_1(N)$. 
%(Note that $W(\chi)= id_K$.) 

Observe that (\ref{eq_char_twist}) implies 
\[
l_\cG(q) \otimes \chi = l_\cG(q) \hspace{4mm} \textup{and} \hspace{4mm} \overline {l_\cG(q)} \otimes \chi = - \overline {l_\cG(q)},
\]
where the second equality follows from $\chi(-1)=-1$.
The pull back of (\ref{eq_zeta_A_B}) 
%%%%%%%%%%%%%%%%%%%%%%%%%%%%%%%%%%%%%%%%%%%    {eq_zeta_A_B_q}
%\begin{equation} \label{eq_zeta_A_B_q}
\[
f:=\zeta(l_\cG(q);\Lambda) - A(\Lambda)l_\cG(q) -B \overline {l_\cG(q)}
\]
%\end{equation}
is a (manifestly non-meromorphic) function on $X_1(N)$. 
%We now twist the function (\ref{eq_zeta_A_B_q}) by $\chi$ and add it to itself to conclude that
We conclude that
\[
f+f \otimes \chi = \zeta (l_\cG(q);\Lambda) + \zeta (l_\cG(q);\Lambda) \otimes \chi -2  A(\Lambda)l_\cG(q)
\]
is a meromorphic function on $X_1(N)$. A standard bounded denominators argument 
%(one employs \cite[Theorem 3.52]{Shimura_arith}) 
implies that, as a formal power series in $q$,
\begin{equation} \label{eq_zeta_q}
\zeta (l_\cG(q);\Lambda) + \zeta (l_\cG(q);\Lambda) \otimes \chi -2  A(\Lambda)l_\cG(q) \in \frac{2}{q} + \gr \llbracket q \rrbracket.
\end{equation}
\begin{rem}

The bounded denominators argument claims that a function $h$ on $X_1(N)$, defined over an algebraic number field $H$, has a $q$-expansion with bounded denominators: 
$h \in \gr \llparenthesis q \rrparenthesis$. 

It follows from the consideration of algebraic models for $X_1(N)$ (see e.g. \cite[Proposition 7.5.2 and Theorem 7.7.1]{DS}) that $h \in H(j,f_1)$, where 
$ j=q^{-1}+744+196884q+21493760q^{2}+\cdots \in \Z  \llparenthesis q \rrparenthesis$ is the $j$-function, and $f_1$ (this notation is borrowed from \cite[Section 7.5]{DS} for the reader's convenience) is a certain function on $X_1(N)$.  For the bounded denominators conclusion, it suffices to show that the $q$-expansion (in fact, $q^{1/N}$-expansion) of $f_1$ has algebraic coefficients which are $p$-integral for almost all primes $p$. 
The function $f_1$ is the weight $2$ Eisenstein series (see \cite[Section 4.6]{DS}) $\wp(1/N; \langle 1, \tau \rangle)$ on $\Gamma_1(N)$ times the ratio $E_4(\tau)/E_6(\tau)$ of the Eisenstein series on $SL_2(\Z)$ of weights $4$ and $6$ with well-known $q$-expansions. Now an inspection of the $q$-expansion of $\wp(1/N; \langle 1, \tau \rangle)$  explicitly calculated in \cite[Section 4.6]{DS} finishes the argument.
 
%An alternative way to conduct the argument is to multiply $h$ by a finite product  $\Phi:= \prod_m(j-s_m)$, where $s_m \in \bar{\Q}$ are the values of $j$ at the poles of $h$ so that $h\Phi$ now has poles only at cusps of $X_1(N)$. Multiplying now by an appropriate power $\Psi = \Delta^r$ of $\Delta = q \prod_{n \geq 1} (1-q^n) \in S_{12}(SL_2(\Z))$, 
%one ensures that $h \Phi \Psi \in S_{12r}(\Gamma_1(N)$ is a holomorphic cusp form. Note that, for the $q$-expansion, we have that $h \Phi \Psi \in H(s_m) \llbracket q \rrbracket$.
%We now apply \cite[Theorem 3.52]{Shimura_arith} to conclude that, possibly after multiplication by some rational integer, all $q$-expansion coefficients of the cusp form  $h \Phi \Psi $ become algebraic integers. Finally, undoing the multiplications $h =  (h \Phi \Psi)\Phi^{-1} \Psi^{-1}$ yields  $h \in \gr \llparenthesis q \rrparenthesis$ (the finitely many primes $p|M$ are  caused by the presence of the finitely many algebraic numbers $s_m$).

\end{rem}

For a prime $p$, observe the identity 
\[
(l_\cG^\phi(t^p)) \otimes \chi = \chi(p)(l_\cG^\phi \otimes \chi) (t^p),
\]
which holds true for any formal power series in $t$, in particular for $l_\cG^\phi \in H\llbracket t \rrbracket$.
%Recall now the definition (\ref{eq_zeta_lambda_mu})  of $A^{(\gp)}(\Lambda)$ (see also (\ref{eq_def})). 
Twist  (\ref{eq_zeta_lambda_mu})  with $\chi$ and add it to itself taking into account that $\mu=0$ whenever $p$ splits in $K$ 
(and keeping in mind  the identity $l_\cG \otimes \chi = l_\cG$ implied by (\ref{eq_char_twist})) to conclude that 
\[
\zeta (l_\cG(t);\Lambda) + \zeta (l_\cG(t);\Lambda) \otimes \chi -2 A^{(\gp)}(\Lambda))l_\cG(t) \in \frac{2}{t} + \cO_\gp \llbracket t \rrbracket.
\]
Combine this observation with (\ref{eq_zeta_q}) (we now regard both  $t$ and $q$ as variables in formal power series) to 
conclude that for all but possibly finitely many primes $\gp$
\[
(A(\Lambda)- A^{(\gp)}(\Lambda)) l_\cG(t) \in \cO_\gp \llbracket t \rrbracket.
\]
Since, for all but finitely many primes $\gp$, the formal group law $\cG$  is isomorphic to $\hat{E}$, and the latter has a reduction of finite height, the formal power series $l_\cG(t)$ has infinitely many zeros. It follows that we must have $A(\Lambda) = A^{(\gp)}(\Lambda)$ as required.

\end{proof}

%%%%%%%%%%%%%%%%%%%%%%%%%%%%%%%%%%%%%%%%%%%%%%%%
%%%%%%%%%%%%%%%%%%%%%%%%%%%%%%%%%%%%%%%%%%%%%%%%
%%%%%%%%%%%%%%%%%%%%%%%%%%%%%%%%%%%%%%%%%%%%%%%% 
%%%%%%%%%%%%%%%%%%%%%%%%%%%%%%%%%%%%%%%%%%%%%%%%
%%%%%%%%%%%%%%%%%%%%%%%%%%%%%%%%%%%%%%%%%%%%%%%%

%\iffalse 

% Appendix section which provides and explains the code starts here
% that is in the arXiv but not in the journal version

%%%%%%%%%%%%%%%%%%%%%%%%%%%%%%%%%%%%%%%%%%%%%%%%
%%%%%%%%%%%%%%%%%%%%%%%%%%%%%%%%%%%%%%%%%%%%%%%%
%%%%%%%%%%%%%%%%%%%%%%%%%%%%%%%%%%%%%%%%%%%%%%%% 
%%%%%%%%%%%%%%%%%%%%%%%%%%%%%%%%%%%%%%%%%%%%%%%%
%%%%%%%%%%%%%%%%%%%%%%%%%%%%%%%%%%%%%%%%%%%%%%%%
%%%%%%%%%%%%%%%%%%%%%%%%%%%%%%%%%%%%%%%%%%%%%%%%

\section{Appendix} \label{sec_app}

In this appendix, we follow a suggestion by the referee and include the GP 
%\cite{PARI2} 
code for the calculations presented in Section \ref{sec_ex}.

We start with the calculations related to the elliptic curve $E$ itself.

\begin{verbatim}
? d=-15;
\end{verbatim}

The ring of integers in $\Q(\sqrt{d})$ has a basis of $\{1,w\}$, where $w$ can be calculated as \verb|quadgen(d)|.
The $j$-nvariant of $E$ must be \verb| ellj(quadgen(d)) |. We can check that indeed $j$ is a root of the polynomial 

\begin{verbatim}
? polclass(d)
%12 = x^2 + 191025*x - 121287375
\end{verbatim}
generating (over $\Q(\sqrt{-15})$) the Hilbert class field $H$ for the  discriminant $d=-15$:

\begin{verbatim}
? subst(polclass(d),x,ellj(quadgen(d))) 
%11 = 6.326664459872514679 E-28
\end{verbatim}

However, working with the polynomial $x^2 + 191025x - 121287375$  is not very convenient, and we prefer a polynomial with smaller coefficients defining the same number field (over $\Q(\sqrt{-15})$).

\begin{verbatim}
? poly=polredbest(polclass(d))
%13 = x^2 - x - 1
\end{verbatim}

We thus see that $H=\Q(\sqrt{d})(j)=\Q(\sqrt{-15},\sqrt{5})$ (with $\text{Gal}(H/\Q(\sqrt{d})) \simeq \text{Cl}(\Q(\sqrt{d})) \simeq \Z/2\Z$).

We now find $j$ as an element of $\Q(\sqrt{5})$ using the embedding $\Q(j) = \Q(\sqrt{5}) \hookrightarrow H$:

\begin{verbatim}
v=nfisincl(polclass(d),polredbest(polclass(d)))[1];
? j=Mod(1,x^2 - x - 1)*v;
? j
%5 = Mod(-85995*x - 52515, x^2 - x - 1)
\end{verbatim}

We can check that our $j$ is indeed a root of the class polynomial:
\begin{verbatim}
? subst(polclass(d),x,j)
%22 = Mod(0, x^2 - x - 1)
\end{verbatim}

We now initiate the data for an elliptic curve over $\Q(\sqrt{5})$ whose invariant equals $j \in \Q(\sqrt{5})$ (this field is initiated as \verb|bnfinit(poly)|).

\begin{verbatim}
bnf=bnfinit(poly); 
EE=ellinit(ellfromj(Mod(1,poly)*v),bnf); 
\end{verbatim}

While this elliptic curve structure has the correct $j$-invariant

\begin{verbatim}
? EE.j
%27 = Mod(-85995*x - 52515, x^2 - x - 1)
\end{verbatim}

it is far from being optimal for further calculations. In particular, the norm of its discriminant calculated as \verb|factor(norm(EE.disc))| \\ equals $2^{24}\ 3^{42}\ 5^6\ 7^{12}\ 11^{12}$.
We  reduce this elliptic curve  manually: 

\begin{verbatim}
? E=ellinit([EE.a4/77^2/3^6/2^2,EE.a6/77^3/3^9/2^3],bnf);
\end{verbatim}
and now   \verb|factor(norm(E.disc))| equals $2^{12} \ 3^6 \ 5^6$ which may be more convenient and possibly speeds up the calculations.

Instead of the $g_2$ and $g_3$ which we use in \eqref{eq_we_example}, GP calculates the $c$-invariants such that $g_2=c_4/12$ and $g_3 = c_6/216$.
Specifically,
\begin{verbatim}
? E.c4/12
%61 = Mod(11505*x + 7110, x^2 - x - 1)
? E.c6/216
%63 = Mod(356720*x + 220465, x^2 - x - 1)
\end{verbatim}
are the coefficients in our model \eqref{eq_we_example}.

In order to obtain the values of $\tau$ (points on the upper half-plane corresponding to the elliptic curve), we use the GP function \verb|E.omega| which gives us two pairs of periods $(\omega_1, \omega_2)$ which yield two values of $\tau$ corresponding to the two embeddings $\Q(j) = \Q(\sqrt{5}) \hookrightarrow \C$:
\begin{verbatim}
? peri=ellperiods(E.omega[1]);
? peri[1]/peri[2]
%79 = 0.25000000000000000000000000000000000000 
+ 0.96824583655185422129481634994559990271*I
? peri=ellperiods(E.omega[2]);
? peri[1]/peri[2]
%81 = -0.50000000000000000000000000000000000000 
+ 1.9364916731037084425896326998911998054*I
\end{verbatim}
Since we know that $\tau \in \Q(\sqrt{-15})$, we easily find the exact values $\tau=(1+\sqrt{-15})/4$ and $\tau=(-1+\sqrt{-15})/2$.
Note that GP labels the periods in a way opposite to that accepted in our paper. Specifically, while we, following \cite{Katz_RAnEis}, assume that $\Im(\omega_2/\omega_1)>0$, 
in the above code, $\Im($ \verb|peri[1]/peri[2]|$) >0$ (i.e. the roles of $\omega_1$ and $\omega_2$ are integchanged).

In order to distinguish the two field embeddings, we may compare the values of the $j$-invariant corresponding to the two pairs of periods
\begin{verbatim}
? ellj(E.omega[1][1]/E.omega[1][2]) - 
                    subst(lift(E.j),x,polroots(x^2-x-1)[1])
%87 = -4.814824860968089633 E-35 + 0.E-33*I
? ellj(E.omega[2][1]/E.omega[2][2]) - 
                    subst(lift(E.j),x,polroots(x^2-x-1)[2])
%88 = 6.162975822039154730 E-33 + 4.885593982039353826 E-33*I
\end{verbatim}
while the approximate values of the two roots \verb|polroots(x^2-x-1)[1]| and \verb|polroots(x^2-x-1)[2]| of the polynomial are clear.

We are now ready to calculate the value of $E_2^*(\tau)$. We calculate approximately the value at $E_2(\tau)$ as a complex number in a very straightforward way, 
write $\Omega=\omega_2/2\pi i$, 
and approximate $\Omega^{-2}E_2$ since the value in question must belong to $\Q(j)=\Q(\sqrt{5})$. The apparent precision of this approximation allows us to claim the exact value. 
We do that for both pairs of periods, and, since we write the results as an algebraic number, it does not depend on the embedding:

\begin{verbatim}
? Z(q)=sum(n=1,130,q^n*sumdiv(n,X,X)); 
? {for(i=1,2, 
peri=ellperiods(E.omega[i]);
tau=peri[1]/peri[2];
T=2*Pi*I/peri[2];  
qq=exp(2*Pi*I*tau);
E2=1-24*Z(qq) - 3/Pi/imag(tau); 
E2a=lift(bestapprnf(E2*T^2,x^2 - x - 1,polroots(x^2-x-1)[i])); 
print(E2a, " ... to check up: ...", 
           subst(E2a, x,polroots(x^2-x-1)[i]) -E2*T^2 ));}
-126*x - 78 ... 
   to check up: ...-9.477423203504693033 E-38 + 0.E-36*I
-126*x - 78 ... 
   to check up: ...3.009265538105056020 E-36 + 0.E-36*I
\end{verbatim}

From this point on, we need no more approximate calculations. It will be convenient to introduce the variable
\begin{verbatim}
w=quadgen(5);
\end{verbatim}
for a root of the polynomial $x^2-x-1$. We will use this variable in the elliptic curve data so that the quantities in $\Q(\sqrt{5})$ appear more readable as \verb| b + a*w  |  instead of 
\verb | Mod(a*x + b, x^2 - x - 1) | (with $a,b \in \Q$).

\begin{verbatim}
e=ellinit([subst( E.a4.pol,x,w), subst( E.a6.pol,x,w)]);
\end{verbatim}

Note that the elliptic curve data in GP for \verb|e| and \verb|E| are different.

We now calculate the first $500$ terms of the Laurent expansion for the Weierstrass $\wp$-function:
\begin{verbatim}
? \ps 500
   seriesprecision = 500 significant terms
? pw=ellwp(e,z);
? pw+O(z^5)
%15 = z^-2 + (711/2 + 2301/4*w)*z^2 
               + (31495/4 + 12740*w)*z^4 + O(z^5)
\end{verbatim}
This number ($500$) for the series precision suffices for the demonstration. Doing some of the calculations below while possible for, say $1000$, becomes pushing the envelope a bit unless one uses a powerful server. 

Laurent expansion of the Weierstrass $\zeta$-function is then calculated as 
\begin{verbatim}
zw=-intformal(pw,z);
? zw+O(z^7)
%18 = z^-1 + (-237/2 - 767/4*w)*z^3 
             + (-6299/4 - 2548*w)*z^5 + O(z^7)
\end{verbatim}
matching the development indicated in Section \ref{sec_ex}.

The logarithm of formal group law corresponding to this elliptic curve is calculated as 
\begin{verbatim}
? eg= ellformallog(e, 1000,'t);
? eg+O(t^8)
%25 = t + (-711 - 2301/2*w)*t^5 
             + (-94485/4 - 38220*w)*t^7 + O(t^8)
\end{verbatim}

To make sure that we are on the right track, we may substitute the formal group law logarithm into the Weierstrass $\wp$-function for $z$.
Since the reduction of our model of the elliptic curve exists at every prime $p>2$, and the Weierstrass $\wp$-function is a function on the curve, we expect that $2$ is the only prime which may appear in the denominators of the coefficients of the series produced. That is indeed the case up to the precision of our calculations:

% \ps 500, not really 1000

\begin{verbatim}
? pwq=subst(pw,z,eg);
? ##
  ***   last result: cpu time 13min, 32,643 ms, 
                    real time 13min, 40,045 ms.
? for(i=1,length(pwq)-5,co=norm(polcoeff(pwq,i)); 
                   den=denominator(co); 
                   den=den/2^valuation(den,2);
                   if(den==1,,print(i, " ... ",den)))
? 
\end{verbatim}

We now plug in the formal group logarithm for $z$ into the Weierstrass $\zeta$-function:
\begin{verbatim}
? zwq=subst(zw,z,eg);
? ##
  ***   last result: cpu time 14min, 29,911 ms, 
                          real time 14min, 35,634 ms.
\end{verbatim}
Using \verb|valuation(zwq,p)| for various primes $p>5$, one can see that  these primes do in fact  appear in the denominators of the coefficients of this series.

We now subtract the formal group logarithm times the factor $(126w+78)/12=(21w+13)/2$ calculated above.
Then we observe that all primes with $\left( \frac{d}{p} \right)=1$ disappear from the denominators of the series:
\begin{verbatim}
? che=zwq-(21*w+13)/2*eg;
? {for(i=1,length(che)-3,co=polcoeff(che,i); 
den=norm(denominator(co)); 
den=den/2^valuation(den,2);
if(den==1,,def=factor(den);
for(j=1,#def~,
if(kronecker(d,def[j,1])==1,
print(i, " ... ",den," ... ",kronecker(d,i))))));
}
? 
\end{verbatim}

However, of course, the prime $p>2$ such that  $\left( \frac{d}{p} \right)=-1$ are still there:
% these valuations are for lim=500, they may be greater for 1000
\begin{verbatim}
? valuation(che,7)
%62 = -2
? kronecker(d,7)
%63 = -1
? kronecker(d,11)
%64 = -1
? kronecker(d,13)
%65 = -1
\end{verbatim}

Let us consider the prime $p=7$ (which is the only one which allows for a coefficient whose $7$-adic valuation is $-2$ within the range of our calculations). 
The point of the paper is that every such prime can be eliminated from the denominators by subtracting an appropriate (depending on the specific prime) multiple of the series 
\begin{verbatim}
? eg2=subst(conj(eg),t,t^7);
? eg2+O(t^50)
%97 = t^7 + (-3723/2 + 2301/2*w)*t^35 
                       + (-247365/4 + 38220*w)*t^49 + O(t^50)
\end{verbatim}
without altering our already established coefficient of $(21w+13)/2$.

We start with finding a coefficient with the smallest (negative) $7$-adic valuation:
\begin{verbatim}
? {
for(i=1,length(che)-5,co=polcoeff(che,i); 
val=valuation(co,7);
if(val==-2,print(i, " ... ",val)))
}
343 ... -2
\end{verbatim}

We also have that
\begin{verbatim}
? Mod(1,49)*(polcoeff(che,343)/polcoeff(eg2,343)*7)
%99 = Mod(47, 49) + Mod(0, 49)*w
\end{verbatim}
Thus, we can simply take $47/7$ as the coefficient:
\begin{verbatim}
? che7=zwq-(21*w+13)/2*eg - 47/7*eg2;
\end{verbatim}
and we indeed find that
\begin{verbatim}
? valuation(che7,7)
%101 = 0
\end{verbatim}
as predicted.

%%%%%%%%%%%%%%%%%%%%%%%%%%%%%%%%%%%%%%%%%%%%%%%%
%%%%%%%%%%%%%%%%%%%%%%%%%%%%%%%%%%%%%%%%%%%%%%%%
%%%%%%%%%%%%%%%%%%%%%%%%%%%%%%%%%%%%%%%%%%%%%%%% 
%%%%%%%%%%%%%%%%%%%%%%%%%%%%%%%%%%%%%%%%%%%%%%%%
%%%%%%%%%%%%%%%%%%%%%%%%%%%%%%%%%%%%%%%%%%%%%%%%

%\fi

% Appendix section which provides and explains the code ends here
% that is in the arXiv but not in the journal version

%%%%%%%%%%%%%%%%%%%%%%%%%%%%%%%%%%%%%%%%%%%%%%%%
%%%%%%%%%%%%%%%%%%%%%%%%%%%%%%%%%%%%%%%%%%%%%%%%
%%%%%%%%%%%%%%%%%%%%%%%%%%%%%%%%%%%%%%%%%%%%%%%% 
%%%%%%%%%%%%%%%%%%%%%%%%%%%%%%%%%%%%%%%%%%%%%%%%
%%%%%%%%%%%%%%%%%%%%%%%%%%%%%%%%%%%%%%%%%%%%%%%%
%%%%%%%%%%%%%%%%%%%%%%%%%%%%%%%%%%%%%%%%%%%%%%%% 

\end{document}